\documentclass[a4paper,11pt]{elsarticle}
\usepackage{amsthm}
\usepackage{amsmath}
\usepackage{amssymb}
\usepackage{enumerate}
\usepackage{graphics}
\usepackage{subfig}
\usepackage{tikz}

\makeatletter
\def\ps@pprintTitle{%
  \let\@oddhead\@empty
  \let\@evenhead\@empty
  \def\@oddfoot{\reset@font\hfil\thepage\hfil}
  \let\@evenfoot\@oddfoot
}
\makeatother

\numberwithin{equation}{section}
\newtheorem{theorem}{Theorem}[section]
\newtheorem{prop}[theorem]{Proposition}
\newtheorem{lemma}[theorem]{Lemma}
\newtheorem*{KL}{Theorem 2.1$^*$}
\newtheorem{corollary}[theorem]{Corollary}
\theoremstyle{definition}
\newtheorem{definition}[theorem]{Definition}
\theoremstyle{remark}
\newtheorem{remark}[theorem]{Remark}

\newcommand\junk[1]{}

\def\C{{\mathfrak{C}}}
\def\A{\mathcal{A}}

\def\D{{\mathcal D}}
\def\E{\mathcal{E}}
\def\Ha{\mathcal{H}}
\def\P{\mathcal{P}}
\def\Q{\mathcal{Q}}
\def\L{\mathcal{L}}
\def\R{{\mathbb R}}
\def\U{\mathcal{U}}
\def\qed{\hfill$\Box$}
\def\Delta{\Gamma^-}
\def\nabla{\Gamma^+}

\begin{document}
\begin{frontmatter}
\title{AZ-identities and Strict 2-part Sperner Properties of Product  Posets\tnoteref{gen}}
\author[bie]{Harout Aydinian\fnref{ayd}}
\author[renyi]{P\'eter L. Erd\H os\fnref{elp}}
\address[bie]{Department of Mathematics, University of Bielefeld,  POB 100131,
    D-33501, Bielefeld \\
    {\tt email}: ayd@math.uni-bielefeld.de}
\address[renyi]{Alfr\'ed R{\'e}nyi Institute, Re\'altanoda u 13-15 Budapest, 1053 Hungary\\
        {\tt email}: elp@renyi.hu}
\fntext[ayd]{This author was supported by DFG-project AH46/7-1 ''General Theory of Information Transfer``}
\fntext[elp]{During this research PLE  enjoyed the hospitality of the AG Genominformatik of
        Universit\"at  Bielefeld, Germany  and was supported in part by the Alexander von
        Humboldt Foundation  and by the Hungarian NSF, under contract NK 78439 and K 68262.}
\tnotetext[gen]{Part of these results was published in an extended abstract form in the Proc. of Eurocomb 2011.}
\begin{abstract}
One of central  issues in extremal set theory is Sperner's theorem and its generalizations. Among such generalizations is  the best-known LYM (also known as BLYM) inequality and the Ahlswede--Zhang (AZ) identity which surprisingly generalizes the BLYM into an identity. Sperner's theorem  and the BLYM inequality has been also generalized to a wide class of posets. Another direction in this research was the study of more part Sperner systems.
In this paper we derive AZ type identities for regular posets. We also characterize all maximum 2-part Sperner systems for  a wide class of product  posets.
\end{abstract}
\end{frontmatter}
\section{Introduction}\label{sec:intro}

Extremal set theory  began in 1928 with Sperner's seminal result  in \cite{Sp}.  Since then dozens of generalizations were discovered (a comprehensive survey of these results can be found in  the excellent book \cite{E3}.) The best-known generalization of Sperner's theorem
is the BLYM inequality due to Bollob\'as \cite{Bo}, Lubell \cite{Lu}, Meshalkin \cite{Me}, and Yamamoto \cite{Ya}. An elegant result discovered by  Ahlswede and  Zhang (\cite{AZ1,AZ2}) gives equalities for any  subset system, which, in turn, directly infers not only the corresponding BLYM inequalities, but also the strict versions. The Sperner property and related problems for posets have been intensively studied by many authors. Another direction in this research is the study of more  part Sperner systems, which was motivated by an application of a Sperner type result in 1945 by Erd\H os \cite{erdos}.

In Section~\ref{sec:class} we derive AZ type identities for regular posets. We also discuss strict Sperner properties and strict  BLYM inequalities for a subclass of normal posets. In Section~\ref{sec:twopart} we present  a 2-part AZ--identity  for regular posets. Finally, we characterize all maximum size 2-part Sperner systems for products of two strictly normal posets.

\section{Sperner systems and related problems in posets}\label{sec:class}

Let $(\P,\leq)$ be a poset  with a  {\em rank function} $r$. We denote the set of  elements of rank $i$ by $\P_i$ ; $i=0,1,\dots,r(\P)$ where  $r(\P)$ is the maximum rank in $\P$.  Also $N_i(\P)$  will denote  the $i$th Whitney number of $\P$, that is $N_i(\P)=|\P_i|$, and for short we will use  $N_i$ when this does not cause confusion. In the sequel sometimes we associate $\P$ with its Hasse diagram and use the corresponding terminology. In particular  the Hasse diagram can be represented as a series of bipartite graphs $G_{i,i+1}(\P)= \big (\P_i,\P_{i+1}; E(\prec)\big )$ (for $i=0,\ldots,r(\P)-1$) where $\prec$ denotes the cover relation between the consecutive levels, that is $(a,b)\in E(\prec) \hbox{ iff } a \prec b$ (in words: if $b$ covers $a$).

For the {\em lower} (resp. {\em upper}) {\em degree} of an element $a\in \P$   we use the notation $d^-(a)$ (resp. $d^+(a)$).  For an element $a\in \P$ we define  $\Gamma^+ _i(a)=\{x\in \P_i: x\geq a\}$ and   $\Gamma^-_i(a)=\{x\in \P_i: x\leq a\}$.  Similarly, for a subset $A\subset \P$ we define $\Gamma^+_i(A)=\{x\in \P_i: x\geq a ~\mbox{for some}~ a\in A\}$ and $\Gamma^-_i(A)=\{x\in \P_i: x\leq a ~\mbox{for some}~ a\in A\}$. We also put $\Gamma^+(a)=\Gamma^+_{r(a)+1}(a)$ and  $\Gamma^-(a)= \Gamma^-_{r(a)-1}(a)$. In fact, $d^+(a)=|\Gamma^+(a)|$ and $d^-(a)=|\Gamma^-(a)|$. Given a subset $A\subset \P$ we denote by $\U(A)$  the {\em upset} (also called {\em filter}) generated by $A$, that is  $\U(A)=\{b\in \P: b\geq a,~ a\in A\}.$ Similarly is defined the {\em downset} (also called  {\em ideal}) $\D(A)=\{b\in \P: b\leq a,~ a\in A\}$.

A poset $\P$ is called {\em regular} if for every element $a\in \P$ of a given rank, both  the lower and upper degrees of $a$  depend on its rank $r(a)$, but not  on the actual choice of $a$ within the level. For a regular poset $\P$ we use  $d^-_i$ (resp. $d^+_i$) for the lower degree (resp. upper degree) of an element in $\P_i$.

A  poset $\P$ is called {\em normal}  if it satisfies the {\em normalized matching property}, that is for  every $A\subset \P_i$ we have  $|A|/N_i\leq
\left |\Gamma^-_{i-1}(A)\right |/N_{i-1} $. The definition implies that every normal poset $\P$ is {\em graded}, that is all minimal and maximal elements  have rank 0 and $r(\P)$, respectively. Thus, each maximal chain in $\P$ has length $r(\P)+1$. It is easy to see that each regular poset $\P$ is  normal.

A subset $C\subset \P$  with $|C|=k$ is called  a {\em chain} of length $k$ if its members are  pairwise comparable.  A chain is called {\em maximal} if it is not contained in a larger chain. A subset $A\subset \P$ is called an {\em antichain} or a {\em Sperner system} if its members are pairwise incomparable. A subset $A\subset \P$ is called  a {\em $k$-Sperner system} if it contains no chain of length $k+1$. A subset  $A\subset \P$ is called homogeneous if it consists of the union of complete levels. Clearly if $A$ is a homogeneous system of $k$ levels then it is a $k$-Sperner system. A poset is said to have {\em strong Sperner property} if for every $k$ there exists a maximum size homogeneous $k$-Sperner system.  It is said to have {\em strict $k$-Sperner property} if all maximum $k$-Sperner systems are homogeneous.

\subsection{AZ--identities for regular posets}\label{sec:AZ}
\smallskip\noindent
In this subsection  we prove two Ahlswede--Zhang type equalities for regular posets. At first we need the following notation. Given  $A\subset \P$ and $x\in \P$, we denote $A^x=\{a\in A: a\leq x\}$. We also define $W_A(x)=0$ if $A^x=\emptyset$ and $W_A(x)=\left| \Gamma^-(x) \setminus \Gamma^+_{r(x)-1}(A^x)\right |$ otherwise. Finally we introduce the shorthand notation {\bf U-poset} for posets with universal lower and upper bounds.
\begin{theorem} \label{th:AZ1}
Let $\P$ be a regular U-poset and let $A\subset \P$.  Then we have
\begin{equation}\label{eq:AZ1}
\sum_{x\in \P}\frac{W_A(x)}{d^-_{r(x)}N_{r(x)}}=1.
\end{equation}
By convention, for  $\{x\}=\P_0\subset A$ (thus $W_A(x)=0$ and   $d^-_{r(x)}=d^-_0=0$)~ we put~ $\frac{W_A(x)} {d^-_0\cdot 1}:=1$.
\end{theorem}
\begin{proof}
We follow directly the idea of the proof in \cite{AZ1}. Let $\P$ be a regular U-poset and let $n:=r(\P)$. Denote by $\E_A$ the boundary edges of $\U(A)$, that is the edges between $\U(A)$ and $\P\setminus \U(A)$. Furthermore, for any given $u\in \U(A)$ let   $\E_A(u) =\{(u,v)\in \E_A\}$.

Note now that the number of maximal chains  passing through the boundary edges $\E_A$ is equal to the number of all maximal chains. Observe also that for each $x\in \U(A)$ the number of maximal chains passing through the edges  $\E_A(x)$ is $d^+_{r(x)}\cdots d^+_{n-1}\cdot |\E_A(x)|\cdots d^-_{r(x)-1}\cdots d_1^-$. The number of maximal chains in a regular poset $\P$ is $N_0d_0^+d_1^+\cdots d^+_{r(\P)-1}= N_{r(\P)}d^ -_{r(\P)}d^-_{r(\P)-1}\cdots d^-_1$. Hence we have
$$
\sum_{x\in \U(A)} d^+_{r(x)}\cdots d^+_{n-1}\cdots|\E_A(x)|\cdots d^-_{r(x)-1}\cdots d_1^-=1\cdots d_0^+d_1^+\cdots d^+_{n-1}.
$$
Since $\E_A(x)=\emptyset$ for $x\notin \U(A)$, we can rewrite this as
\begin{equation}\label{eq:chains2}
 \sum_{x\in \P}\frac{|\E_A(x)|d_1^-\cdots d^-_{r(x)-1}}{d_0^+d_1^+\cdots d^+_{r(x)-1}}=1.
\end{equation}
Observe  that, in view of regularity, for every $1\leq k\leq n$ the following holds
\begin{equation}\label{eq:levelsize}
N_k=\frac{d^+_0\cdots d^+_{k-1}}{d^-_1\cdots d^-_k}.
\end{equation}
Therefore, (\ref{eq:chains2}) gives
$$
\sum_{x\in \P}\frac{|\E_A(x)|}{d^-_{r(x)}N_{r(x)}}=1.
$$
To finish the proof it remains to show that $|\E_A(x)|=W_A(x).$ The latter is clear since
$\Gamma^+_{r(x)-1}(A^x)\subset \U(A)$ and hence $\E_A(x)=\Gamma^-(x)\setminus \Gamma^+_{r(x)-1}(A^x)$.
\end{proof}

 \bigskip\noindent We  emphasize that  Theorem \ref{eq:AZ1}, as well as all AZ type identities presented below,  can be extended to any regular poset $\P$ as follows:

Given a regular poset $\P$ and a subset $A\subset \P$, let us introduce the function $f: \P\rightarrow \R$, defined as follows:\\
$$
f_A(x)=
  \begin{cases}
    0, & \hbox{if } x\in \P_0\setminus A; \\
    1/N_0, & \hbox{if } x\in \P_0\cap A; \\
    1/N_n, & \hbox{where } n:=r(\P) \mbox{ if } x\in\P_n\setminus U(A)\\
    \frac{W_A(x)}{d^-_{r(x)}N_{r(x)}}, & \hbox{otherwise.}
  \end{cases}
$$
\begin{KL}
For a regular poset $\P$ and a subset $A\subset\P$ we have
$$
\sum_{x\in\P}f_A(x)=1.
$$
\end{KL}
\begin{proof}
To prove the statement, we  add to $\P$ universal maximal and a minimal elements $a$ and $b$, respectively, and then apply
Theorem \ref{th:AZ1} to the  U-poset $\P^*=\P\cup\{a,b\}$.
Observe now that in the identity (\ref{eq:AZ1}), applied for $\P^*$, we have  $\frac{W_A(x)} {d^-_{r(x)} N_{r(x)}}=f_A(x)$
for  $x\in\P\setminus(\P_n\setminus U(A))$. Note also that $\frac{W_A(a)}{d^-_{r(a)}N_{r(a)}}=\frac{W_A(a)}{|\P_n|\cdot 1}=\sum\limits_{x\in \P_n\setminus U(A)}f_A(x)$
 and  $\sum\limits_{x\in \P_n\setminus U(A)}  \frac{W_A(x)} {d^-_{r(x)}N_{r(x)}}=0$. This completes the proof.
\end{proof}

\noindent It is not hard to see that for the Boolean lattice $2^{[n]}$ we have
$$
W_A(x)=  |\bigcap_{a\in A:
a\subseteq x}a|, \quad d^-_{r(x)}=|x|,~ N_{r(x)}=\binom {n}{|x|}
$$
and  (\ref{eq:AZ1}) reduces to the original Ahlswede-Zhang identity (see \cite{AZ1}).  Notice also that if $A$ is an antichain then $W_A(x)=d^-_{r(x)}$ for all $x\in A$
and hence we have the following identity for antichains.
\begin{corollary}\label{th:AZ1A}
For an antichain $A$ in a regular U-poset $\P$ we have
\begin{equation}\label{eq:AZ1A}
\sum_{x\in A}\frac{1}{N_{r(x)}}+ \sum_{x\in \P \setminus A}\frac{W_A(x)} {d^-_{r(x)} N_{r(x)}}=1.
\end{equation}
\end{corollary}
\noindent This in particular implies  the LYM  inequality for regular posets:
\begin{equation}
\sum_{x\in A}\frac{1}{N_{r(x)}}\leq 1
\end{equation}
which seemingly first appeared in Baker \cite{bak}.

Note next that Corollary \ref{th:AZ1A}  can be  extended to $k$-Sperner systems. By the dual version of the Dilworth theorem  (attributed to Erd\H{o}s
and Szekeres see \cite{lovasz} Ex. 9.32b),  we can decompose such a system $A$ into $k$ disjoint  antichains $A=A_1\cup \ldots\cup A_k$.
 Applying (\ref{eq:AZ1A}) to each antichain $A_i$ and then summing up all $k$ identities we get  the following identity which in turn implies $k$-BLYM inequality for regular posets.
\begin{corollary} \label{th:dual}
Let $A$ be $k$-Sperner system in a regular U-poset $\P$. Then we have
$$
\sum_{x\in A}\frac{1}{N_{r(x)}}+ \sum_{i=1}^k\sum_{x\in \P \setminus A_i}\frac{W_{A_{i}}(x)} {d^-_{r(x)}N_{r(x)}}=k.
$$ \qed
\end{corollary}

\begin{remark}
It is important to recognize that Theorem \ref{th:AZ1} does not hold for all normal posets. For example the poset on Figure~\ref{counter} (a) is normal,  however does not have this property. Indeed, take  the antichain $A= \{a,c\}$. Then $W_A(a)=W_A(b)=W_A(c)=1$ and $N_{r(a)}= N_{r(b)}= N_{r(c)}=2,$ furthermore,  $d^-_{r(a)}=d^-_{r(c)}=1$ and  $ d^-_{r(b)}=2.$ Therefore,
$$
\sum_{x\in P}\frac{W_A(x)}{d^-_{r(x)}N_{r(x)}}=1/2 +1/2 + 1/4 >1.
$$
\end{remark}

\medskip\noindent
Ahlswede and Zhang found yet another generalization of  identity (\ref{eq:AZ1}) for the Boolean lattice (see \cite{AZ2}) which in turn is a sharpening of the  Bollob\'as inequality \cite{Bo}. This latter result is based on a higher level regularity of the Boolean lattice,
 therefore we also have to restrict ourselves to a  subclass of regular posets with  stronger regularities.

We call a U-poset $\P$ {\em strongly regular} if for all  pairs $a,b \in \P$, with $a< b$,  the number $|\Gamma^+_i(a)\cap \Gamma^-_j(b)|$ depends only on $i,j$ and  $r(a)$, $r(b)$.

Thus, we can define the quantity $\lambda_i(k,l)$ for all integers $0\leq k\leq l\leq r(\P)$ and $k\leq i\leq l$ as $\lambda_i(k,l)=|\Gamma^+_i(a)\cap \Gamma^-_i(b)|$, ~where $a,b\in \P,~ a\leq b$ and $r(a)=k, r(b)=l$.  For all other triples $(i,k,l)$ of integers we put  $\lambda_i(k,l)=0$. Note that a strongly regular U-poset is regular.
\begin{theorem} \label{th:AZ2}
Let $A=\{a_1,\ldots,a_m\}$ and $B=\{b_1,\ldots,b_m\}$ be subsets of a strongly regular U-poset  $\P$ such that $a_i\leq b_j$ ~if and only if~ $i=j$. Let also $k_i:=r(a_i)$ and $l_i:=r(b_i)$; $i=1,\ldots, m$. Then we have
\begin{equation}
\sum_{i=1}^m \beta(k_i,l_i)+  \sum_{x\in \U(A)\setminus\D(B)} \frac{W_A(x)}{d^-_{r(x)}N_{r(x)}}=1,
\end{equation}
where
$$
\beta(k_i,l_i):=\sum_{j=0}^{l_i-k_i}\frac{\lambda_{k_i+j} (k_i,l_i)} {d^-_{k_i+j}N_{k_i+j}}\cdot \left (d^-_{r(x)}-\lambda_{k_i+j-1}(k_i,k_i+j)\right ).
$$
\end{theorem}
\begin{proof}
For every pair $a,b\in\P$ with $a\leq b$ let us denote $S(a,b)=\{x\in\P:a\leq x\leq b\}$.  By the definition of strong regularity
$$
|S(a,b)|=\sum_{j=0}^{r(b)-r(a)}\lambda_{r(a)+j}(r(a),r(b)).
$$
Note also that (by the condition of the theorem) $S(a_i,b_i)\cap S(a_j,b_j)=\emptyset$ for all $i\neq j$ ($i,j\in [m]$). We proceed now like in the proof of Theorem~\ref{th:AZ1}  counting all maximal chains.

We count first all maximal chains that meet $S(a_i,b_i)$ for  $i=1,\dots,m$. This number (by the same arguments as in the proof of Theorem~\ref{th:AZ1})  equals
$$
\sum_{a_i\leq x\leq b_i}\frac {W_A(x)}{d^-_{r(x)}N_{r(x)}}\cdot d_0^+d_1^+\ldots d^+_{n-1}.
$$
Note that these $m$ sets of  maximal chains are pairwise disjoint. The number of remaining maximal chains equals
$$
\sum_{x\in \U(A)\setminus\D(B)}\frac{W_A(x)}{d^-_{r(x)}N_{r(x)}}\cdot  d_0^+d_1^+\ldots d^+_{n-1}.
$$
Thus, we have
$$
\sum_{i=1}^m\sum_{a_i\leq x\leq b_i}\frac {W_A(x)}{d^-_{r(x)}N_{r(x)}}+ \sum_{x\in \U(A)\setminus\D(B)}\frac{W_A(x)}{d^-_{r(x)}N_{r(x)}}=1.
$$
Observe now that for each $a_i\leq x\leq b_i$ we have $W_A(x)=d^-_{r(x)} -\lambda_{r(x)-1}(k_i,r(x))$. Hence, we get
$$
\sum_{a_i\leq x\leq b_i}\frac {W_A(x)}{d^-_{r(x)}N_{r(x)}}= \beta(k_i,l_i);~  i=1,\ldots,m.
$$
This completes the proof.
\end{proof}
\noindent Again, when we apply this theorem to the Boolean lattice, we get back verbatim the original second AZ--identity (see \cite{AZ2}), since in this case we have
\begin{eqnarray*}
\beta(k_i,l_i)&=& \sum_{a_i\subseteq x\subseteq b_i}\dfrac{|a_i|}{|x|\binom n{|x|}}= \\
&=&\sum_{k=0}^{|b_i|-|a_i|}\binom {|b_i|-|a_i|}k\cdot \frac{|a_i|}{(|a_i|+k)\binom n{|a_i|+k}}=
\dfrac 1{ \binom{n-|b_i|+|a_i|}{|a_i|}}.
\end{eqnarray*}

\subsection{Strict BLYM inequalities for normal posets}\label{sec:strict}
\smallskip\noindent
In this subsection we turn our attention to  {\em strict} Sperner theorems and {\em strict} BLYM inequalities. We  start with a strong  result of Kleitman (which we use later), which in fact extends the BLYM inequality to a wide class of posets. The following notion was introduced in Kleitman \cite{Kl}.
\begin{definition}
If $\C(\P)$ denotes the set of maximal chains in a graded poset $\P$ then a {\bf regular covering} of $P$ is a function $f: \C(P)\rightarrow \R_+ \cup\{0\}$ such that   the  following  holds:
\begin{equation}\label{eq:cover}
(a) \sum_{C\in\C} f(C)=1; \quad (b) \ \forall a\in P\ :\ \sum_{C\in \C: a\in C}f(C)=\frac{1} {N_{r(a)}}.
\end{equation}
\end{definition}
\begin{theorem}[Kleitman, \cite{Kl}] \label{th:kleitman}
A poset satisfies the $k$-BLYM inequalities for all possible $k$ if and only if it is a  normal poset, which, in turn, is equivalent to the existence of a {\em regular chain covering}.
\end{theorem}
Recall that any regular poset is  normal, so Kleitman's theorem applies for them. It is important to recognize that normality does not always imply  that the poset also satisfies the strict BLYM inequality, or  strict Sperner property. In fact, none of the normal posets on Figure~\ref{counter} satisfy the strict Sperner property.
\begin{figure}[h!]
\centering
\subfloat [] {
\begin{tikzpicture}[scale=0.4]
\filldraw [black,] (3,0) circle (8pt);
\filldraw [black] (1,2.5) circle (8pt);
\filldraw [black] (5,2.5) circle (8pt);
\node at (0,5) {$\mathbf{a}$};
\node at (6,5) {$\mathbf{b}$};
\node at (6,2.5) {$\mathbf{c}$};
\filldraw [black] (1,5) circle (8pt);
\filldraw [black] (5,5) circle (8pt);
\filldraw [black] (3,7.5) circle (8pt);
\draw [line width=1pt]  ( 3,0) -- ( 1,2.5);
\draw [line width=1pt]  ( 1,2.5) -- ( 1,5);
\draw [line width=1pt]  ( 1,5) -- ( 3,7.5);
\draw [line width=1pt]  ( 3,7.5) -- ( 5,5);
\draw [line width=1pt]  ( 5,5) -- ( 5,2.5);
\draw [line width=1pt]  ( 3,0) -- ( 5,2.5);
\draw [line width=1pt]  ( 1,2.5) -- ( 5,5);
\end{tikzpicture}
    }
    \qquad\qquad\qquad
\subfloat[]{
\begin{tikzpicture}[scale=0.4]
\filldraw [black] (3,0) circle (8pt);
\filldraw [black] (1,2.5) circle (8pt);
\filldraw [black] (5,2.5) circle (8pt);
\filldraw [black] (1,5) circle (8pt);
\filldraw [black] (5,5) circle (8pt);
\filldraw [black] (3,7.5) circle (8pt);
\draw [line width=1pt]  ( 3,0) -- ( 1,2.5);
\draw [line width=1pt]  ( 1,2.5) -- ( 1,5);
\draw [line width=1pt]  ( 1,5) -- ( 3,7.5);
\draw [line width=1pt]  ( 3,7.5) -- ( 5,5);
\draw [line width=1pt]  ( 5,5) -- ( 5,2.5);
\draw [line width=1pt]  ( 3,0) -- ( 5,2.5);
\node at (7,0) {{$\P_0$}};
\node at (7,2.5) {{$\P_1$}};
\node at (7,5) {{$\P_2$}};
\node at (7,7.5) {{$\P_3$}};
\end{tikzpicture}
}
\caption{These posets are not strict Sperner.
(a) normal, but not {\em regular} poset, the first AZ--identity does not hold;
(b) normal, but not {\em level-connected} poset, since  the induced bipartite graph $G_{1,2}$ is not connected. Neither poset  satisfies the strict Sperner property.}
\label{counter}
\end{figure}
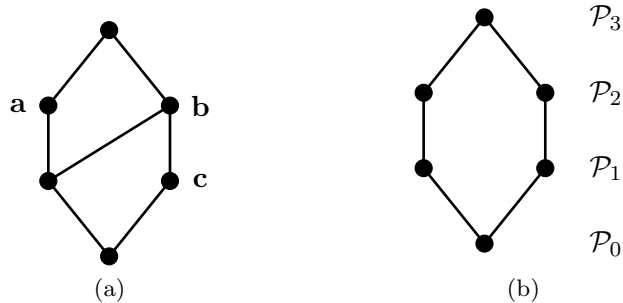

\medskip\noindent The study of the strict Sperner property of normal posets was initiated by Engel (\cite{E1,E2}). He introduced the following strengthening of the normalized matching property:
\begin{definition}
A  normal poset $\P$ is called {\bf strictly normal} if it satisfies the {\bf strict normalized matching property}, that is for  every proper subset $A\subset \P_i$; $i\in\{1,\ldots,r(\P)\}$ we have  $|A|/N_i <\left |\Gamma^- (A) \right |/N_{i-1}$.
\end{definition}
\begin{theorem}[Engel \cite{E1}]\label{th:engel}
Every strictly normal poset has strict $k$-Sperner property. \qed
\end{theorem}
\noindent Furthermore Engel also proved (\cite{E2}, \cite[Ch. 4.6]{E3})  that the direct product of  strictly normal posets with  log-concave Whitney numbers is again a strictly normal posets with  log-concave Whitney numbers. Therefore, Theorem \ref{th:engel} holds for the direct products of such posets.

The following easy result will be used in Section \ref{sec:twopart}.
\begin{theorem} \label{th:strict}
Let $\P$ be a strictly normal poset. Then $\P$  satisfies the  strict $k$-LYM inequality, that is if  $F\subset \P$ is a $k$-Sperner system such that
\begin{equation}\label{eq-LYMeq}
\sum_ {a\in F}\frac 1{N_{r(a)}}=k
\end{equation}
holds, then $F$ is homogeneous.
\end{theorem}
\begin{proof}
By the dual Dilworth theorem, $F$ can be decomposed into $k$ pairwise disjoint antichains. Then equality (\ref{eq-LYMeq}) can be written as the sum of $k$-BLYM inequalities, which in fact are equalities. Let $A$ be one of the antichains in the decomposition of $F$, thus $\sum_ {a\in A}\frac 1{N_{r(a)}}=1$. Suppose now that  $A$  contains elements from different levels and let  $B\subset \P_i$  be the elements of the highest level present in $A$. Replacing $B$  by $\Gamma^-(B)$ in $A$  we get a new Sperner system $A^\prime$. However, in view of strict normality, we have now~ $\sum_ {a\in A^\prime}\frac 1{N_{r(a)}}>1,$ a contradiction with  Theorem \ref{th:kleitman}.
\end{proof}

\medskip\noindent The  strict normalized matching property  seems to be a rather strong prerequisite. How can one find such posets? The following notion (introduced  in \cite{E1},\cite{E2}) leads to a subclass of  regular posets that satisfy this property.
\begin{definition}
A poset $\P$ is called {\bf level-connected} if for each $i$ the bipartite graph $G_{i,i+1} (\P)$ is connected.
\end{definition}
\noindent The following fact is a special case of a more general statement in \cite{E1}.
\begin{lemma} \label{th:strict-normal}
A poset  is strictly normal if it is regular and  level-connected.
\end{lemma}
\noindent Thus, we have the following:
\begin{corollary} \label{th:regular-lev}
Let $\P$ be a regular level-connected poset. Then it  satisfies the  strict $k$-LYM inequality.
\end{corollary}
\begin{remark}
Using the same approach  as in Lemma \ref{lm:4} (in Subsection \ref{sec:2SP}), we can also show that for  a maximum size $k$-Sperner system $F$ in a normal poset the equality $\sum_{a\in F}\frac 1{N_{r(a)}}=k$ holds. Applying now  Theorem \ref{th:strict} for a strictly normal  poset we get Theorem \ref{th:engel}.
\end{remark}

\section{Two-part problems}\label{sec:twopart}
\medskip\noindent
Let $S=S_1 \uplus S_2$ be a fixed partition of the underlying $n$-element set $S$. A family $\Ha$ of subsets of $S$ is called a {\em $2$-part Sperner family } if
$$
\forall E, F \in \Ha \quad \big ( E \subsetneqq F \Rightarrow \forall i : F
\setminus E \not\subseteq S_i \big ).
$$
Katona \cite{katona} and Kleitman \cite{kleitman} proved independently around 1965 that no 2--part Sperner system in the Boolean lattice can be bigger than a maximum size Sperner family.

Much later, only in 1986, P.L. Erd\H{o}s and Katona  \cite{EK2} proved that the 2--part Sperner theorem in the Boolean lattice is strict,   that is all optimal 2--part Sperner families are homogeneous. Here, homogeneous means that the family is a union  of products of complete levels in $2^{S_1}$ and $2^{S_2}$.

In 1971 Schonheim proved (see \cite{schon}) the 2--part Sperner theorem for the direct product of two divisor lattices: here a subset $A \subset \P_1 \times \P_2$ is called  {\em 2--part Sperner system}  if whenever both components of $(x,y)\in A$ are smaller than the corresponding components of $(u,v)\in A$ then neither $x=u$ nor $y=v$ hold.  Katona in \cite{katona1}  generalized this result for the direct product of {\em symmetric chain} posets.

\subsection{2--part AZ--identity}\label{sec:2AZ}
\smallskip\noindent
In this subsection we derive a 2--part AZ type identity for  direct products of regular posets.

Let $\P$ and $\Q$ be regular posets and let $A\subset \P\times \Q$. For each $y\in \Q$ we define $A(y)=\{a\in \P: (a,y)\in A\}$ and $A^x(y)=\{a\in \P: (a,y)\in A, a\leq x\}$. Furthermore, for a  given $(x,y)\in \P\times \Q$ we define $W_{A(y)}(x)=0$ if  $A^x(y)=\emptyset$ and
 $W_{A(y)}(x)= \left |\Gamma^-(x)\setminus  \Gamma^+ _{r(x)-1} (A^x(y))\right |$ otherwise.
\begin{theorem}
For regular U-posets $\P$ and $\Q$ and a subset $A\subset \P\times \Q$ we have
$$
 \sum_{(x,y)\in \P\times \Q}\frac{W_{A(y)}(x)}{d^-_{r(x)} N^{(1)}_{r(x)}N^{(2)}_{r(y)}}=r(\Q)+1.
$$
By convention, for  $\{x\}=\P_0$ and $(x,y)\in A$ (thus  $W_{A(y)}(x)=0, d^-_{r(x)}=d^-_0=0$, $N^{(1)}_{r(x)}=1$)
 we put~ $\frac{W_{A(y)}(x)}{d^-_0\cdot 1\cdot N^{(2)}_{r(y)}}:=\frac{1}{N^{(2)}_{r(y)}}$).
\end{theorem}
\begin{proof}
For any given $y\in \Q_i$ we can apply the AZ--identity (\ref{eq:AZ1}) for the subset $A(y)\subset \P$. Summing up  the corresponding identities for all  $y\in \Q_i$ we get:
$$
\sum_{(x,y)\subset \P\times \Q: y\in \Q_i}  \frac{W_{\A(y)}(x)}{d^-_{r(x)} N^{(1)}_{r(x)}N^{(2)}_{r(y)}}=1.
$$
We do the same for all levels $i=0,1,\ldots,r(\Q)$ in $\Q$ and then sum up all these identities. This gives the result.
\end{proof}
\noindent It it easy to see that the latter implies the following identity for 2-part Sperner systems in a product of two regular U-posets.
\begin{corollary}
Let $\P$ and $\Q$ be regular U-posets. Let also  $N^{(1)}_i:=|\P_i|$, $N^{(2)}_i:= |\Q_i|$, and $r(\Q)\leq r(\P)$. If $A\subset \P\times \Q$ is a 2--part Sperner system,  then
\begin{equation}\label{eq:2-partAZ}
\sum_{(a,b)\in A}\frac 1{N^{(1)}_{r(a)}N^{(2)}_{r(b)}}+ \sum_{(x,y)\notin A} \frac{W_{A(y)}(x)}{d^-_{r(x)}N^{(1)}_{r(x)}N^{(2)}_{r(y)}}=r(\Q)+1.
\end{equation}
\qed
\end{corollary}

\subsection{A strict 2--part Sperner theorem}\label{sec:2SP}
\smallskip\noindent
In this subsection we generalize the strict 2--part Sperner theorem (see \cite{EK2}) for direct products of two  strictly normal posets. The proof follows closely the proof of the original strict 2--part Sperner theorem  in \cite{AE}.
\begin{theorem}\label{th:strict2}
Let   $\P$ and $\Q$ be   strictly  normal posets and let  $F\subset \P\times \Q$ be a maximum  size 2--part Sperner system. Then $F$ is a  homogeneous system.
\end{theorem}
In fact, the theorem allows us to characterize all maximum size 2--part Sperner systems in the described subclasses of normal posets. To do this we define below the notion of the {\em well-paired} homogeneous system.

Let $F=\bigcup\limits_{(i,j)\in I}\P_i\times Q_j$. Then clearly $F$ is a 2--part Sperner system if and only if $I$ is a  {\em transversal}, that is every two members of $I$ do not have a component in common. Clearly $|I|\leq \min\{r(\P),r(\Q)\}+1$ and in the case of equality, $I$ is called a {\em full transversal} (otherwise it is called {\em partial}). Note that if $F$ is a maximum size homogeneous system,  then  $I$  is a full transversal, since every partial
 transversal can be extended to a full one. Now clearly  $|F|=\max\limits_{I}\sum\limits_{(i,j)\in I} |P_i| |Q_j|$ where $I\subset \{0,\ldots,r(\P)\}\times \{0,\dots,r(\Q)\}$ is a full transversal.
Thus, an optimal $F$ consists of union of products of  full levels in each poset, such that  $|I|$ largest levels   in  $\P$ are paired with corresponding  $|I|$ largest levels in $\Q$. We call such a 2--part Sperner system {\em well-paired}. In fact, we have seen that the following holds.
\begin{lemma} \label{lm:3}
Let $\P,\Q$ be  ranked posets and let $F\subset \P\times \Q$ be a maximum size homogeneous 2--part Sperner system. Then $F$ is a well paired  system.
\end{lemma}
\noindent We can rephrase now Theorem \ref{th:strict2} as follows:
\begin{description}
\item[$(\dag)$] \quad {\em All maximum size 2--part Sperner systems in the described classes of product posets are well--paired.}
\end{description}

\noindent We prove  Theorem~\ref{th:strict2}  through a series of lemmas. Let $\C:=\C_1 \times \C_2$  where $\C_1:=\C(\P), \C_2:=\C(\Q)$ and let  $n_1:=r(\P)$, $n_2:=r(\Q)$.
\begin{lemma} \label{lm:2}
Let $\P,\Q$ be normal posets and let  $F\subset \P\times \Q$ be a 2--part Sperner system. Let also  $(C_1,C_2)\in \C(\P)\times  \C(\Q)$. If $A:=F\cap(C_1\times C_2)$ then
\begin{equation}\label{eq:upper}
|A|\leq \min \{r(\P),r(\Q)\}+1.
\end{equation}
\end{lemma}
\begin{proof}
Let $C_1$ and $C_2$ consist of elements $a_0<a_1<\ldots < a_{n_1}$ and $b_0< b_1 <\ldots < b_{n_2}$ respectively. Let $I=\{(i,j)\in (n_1+1)\times (n_2+1): (a_i,b_j)\in A\} $. Observe now that this $I$ is a transversal. Hence  the result.
\end{proof}
\medskip\noindent  The following result plays a key role in our proof.
\begin{lemma} \label{lm:4}
Let $\P, \Q$ be normal posets and let  $F\subset \P\times \Q$ be a maximum size 2--part Sperner system. Assume also that $n_2 \le  n_1.$ Then
\begin{equation}
\sum_{(a,b)\in F}\frac1{N^{(1)}_{r(a)}N^{(2)}_{r(b)}}=n_2+1.
\end{equation}
\end{lemma}
\noindent
\begin{remark}  The prove of this equality from \cite{AE} for Boolean lattices  used  results from the theory of convex hulls for 2--part Sperner families (\cite{EK3}). Here we could apply  similar arguments, using in addition a result of Sali (\cite{Sa}). For the sake of completeness,  we give here a direct proof of this statement.
\end{remark}
\begin{proof}
Let $f_1,f_2$ be regular chain coverings of $\P$ and $\Q$ respectively.  Define $g(C_1,C_2) =  f_1(C_1)f_2(C_2)$ for all $(C_1,C_2)\in  \C_1\times \C_2$. Let $F\subset \P\times \Q$ be a maximum size 2--part Sperner system. Then in view of regular chain coverings $f_1$ and $f_2$ we have
\begin{eqnarray}
|F|&=&\sum_{(a,b)\in F}1=\sum_{(a,b)\in F}\left (\sum_{\genfrac{}{}{0pt}{}{C_1\in \C_1:}{a\in C_1}}f_1(C_1)N^{(1)}_{r(a)}\right )  \left (\sum_{\genfrac{}{}{0pt}{}{C_2\in \C_2:}{a\in C_2}}f_2(C_2) N^{(2)}_{r(a)}\right ) \nonumber\\
&=&\sum_{(a,b)\in F}~\left (\sum_{\genfrac{}{}{0pt}{}{(C_1,C_2)\in \C:} {(a,b)\in(C_1\times C_2)}} g(C_1,C_2)N^{(1)}_{r(a)}N^{(2)}_{r(b)}   \right )\nonumber \\
&=&\sum_{(C_1,C_2)\in \C} \left ( g(C_1,C_2) \sum_{(a,b)\in (C_1\times C_2)\cap F}
N^{(1)}_{r(a)}N^{(2)}_{r(b)}  \right )\nonumber \\
&\leq&\sum_{(C_1\times C_2)\cap F\neq \emptyset}g(C_1,C_2)\left(\max_{(C_1,C_2)\in
\C} \sum_{(a,b)\in(C_1\times C_2)\cap F} N^{(1)}_{r(a)}N^{(2)}_{r(b)} \right)
\label{eq:one}\\
&\leq & \max_{(C_1,C_2)\in\C}  \sum_{(a,b)\in (C_1\times C_2)\cap F} N^{(1)}_{r(a)} N^{(2)}_{r(b)}.\label{eq:two}
\end{eqnarray}

\noindent Before we continue the proof of Lemma~\ref{lm:4} we notice that the last inequality together with Lemma \ref{lm:3} implies the following:
\begin{prop}\label{co:homo}
In the product of two normal posets there exist maximum size 2--part Sperner systems  which are  well-paired $($homogeneous$)$ systems.
\end{prop}

\noindent Now, since  $F$ has maximum size we must have  equalities in (\ref{eq:one}) and (\ref{eq:two}). From this we can infer an important consequence:
\begin{prop}\label{lm:2exact}
For any maximum size 2--part Sperner system  $F\subset \P\times \Q$ we have $|(C_1\times C_2)\cap F|=n_2+1$ for all $(C_1,C_2)\in \C_1\times \C_2$ with  $g(C_1,C_2)>0$.
\end{prop}
\noindent {\em Proof.}
Indeed,  the equality in (\ref{eq:one}) implies that for each pair $(C_1,C_2)$
with $(C_1\times C_2)\cap F\neq\emptyset$ we  have
$$
\sum_{(a,b)\in (C_1\times C_2)\cap F} N^{(1)}_{r(a)}N^{(2)}_{r(b)}=|F^*|,
$$
where $F^*$ is an optimal homogeneous  2--part Sperner system. Since, in view of
Lemma \ref{lm:3}, $F^*$ is full, it implies  that  $|\{(a,b)\in (C_1\times C_2) \}|=n_2+1$. On the other hand, an equality in (\ref{eq:two}) implies that
\begin{equation}\label{eq:lm:2exact}
\sum_{(C_1\times C_2)\cap F\neq \emptyset}g(C_1,C_2)=1,
\end{equation}
which means that $(C_1\times C_2)\cap F\neq \emptyset$ for each pair $(C_1,C_2)$  with
$g(C_1,C_2)>0$. This completes the proof of Proposition~\ref{lm:2exact}.
\end{proof}

\smallskip\noindent  Now we are going to finish the proof of Lemma~\ref{lm:4}. We have
\begin{eqnarray}\label{lm:exact}
\sum_{(a,b)\in F}\frac{1}{N^{(1)}_{r(a)}N^{(2)}_{r(b)}} & = & \sum_{(a,b)\in F}\left  (\sum_{\genfrac{}{}{0pt}{}{(C_1,C_2)\in\C:} {(a,b)\in(C_1\times C_2)}}g(C_1,C_2) N^{(1)}_{r(a)} N^{(2)}_{r(b)})\right ) \frac{1}{N^{(1)}_{r(a)}N^{(2)}_{r(b)}}\nonumber  \\
& = &\sum_{(a,b)\in F}~\sum_{\genfrac{}{}{0pt}{}{(C_1,C_2)
\in\C:}{(a,b)\in(C_1\times C_2)}} g(C_1,C_2) \nonumber \\
& = & \sum_{(C_1,C_2)\in\C}g(C_1,C_2)~ \left (\sum_{ (a,b)\in (C_1\times C_2)\cap
F}1 \right )\label{l2} \\
& = & (n_2+1)\left (\sum_{(C_1\times C_2)\cap F\neq \emptyset}g(C_1,C_2) \right
)=n_2+1. \label{l4}
\end{eqnarray}
Here in the first line we applied the Definition \ref{eq:cover}(b). The second sum in equality (\ref{l2}) is $n_2+1$, due to Proposition \ref{lm:2exact}, finally the sum in equality (\ref{l4}) is $1$ because of equality (\ref{eq:lm:2exact}). We are done with Lemma~\ref{lm:4}.
\qed$_{\ref{lm:4}}$

\smallskip
\begin{proof}[Proof of Theorem \ref{th:strict2}]
Let $F$ be a maximum size 2--part Sperner system, and let $b\in \Q.$ Furthermore, let $F(b):= \{ a\in \P : (a,b) \in F\}.$ Then
\begin{equation}\label{eq:old}
\sum_{a\in F(b)} \frac{1}{N^{(1)}_{r(a)}} \le 1
\end{equation}
since, by definition, $F(b)$ is a Sperner system in $\P.$ Then
\begin{displaymath}
\sum_{b'\in \Q_{r(b)}} \sum_{a\in F(b')}  \frac{1}{N^{(1)}_{r(a)}} \frac{1}{N^{(2)}_{r(b)}} \le 1,
\end{displaymath}
finally taking all possible levels from $\Q$ we have
\begin{displaymath}
\sum_{i=0}^{n_2}\sum_{b\in \Q_{i}} \sum_{a\in F(b)} \frac{1}{N^{(1)}_{r(a)}} \frac{1}{N^{(2)}_{r(b)}} \le n_2+1.
\end{displaymath}
This can be rewritten as
\begin{equation}\label{eq:old1}
\sum_{(a,b)\in F} \frac{1}{N^{(1)}_{r(a)}N^{(2)}_{r(b)}} \le n_2+1,
\end{equation}
which is exactly the same, as inequality (\ref{l4}). Since $F$ is optimal, by Lemma \ref{lm:exact}, inequality  (\ref{eq:old1}) holds with equality, and the same holds for inequality (\ref{eq:old}). (Note that for a regular poset (\ref{eq:old1}) follows from (\ref{eq:2-partAZ}).)

Now Theorem \ref{th:strict}   implies that for each $b\in \Q$ the  family  $F(b)$  is a full level in $\P$. If $n_1=n_2$ then we can repeat our reasoning, exchanging the roles of $\P$
and $\Q$ which clearly finishes the proof.  However, if $n_2 < n_1$, then we have to work a little bit more - but the remaining part is about only numbers, therefore the same method what was used in the proof of Theorem 1  in  \cite{AE} finishes the proof of Theorem \ref{th:strict2}.
\end{proof}

\bigskip\noindent
In conclusion we list  some  classes of product posets for which Theorem \ref{th:strict2} holds,
 namely, the products of any two posets presented below.

\smallskip\noindent
The first four classes of posets listed below are regular and level-connected:
\begin{enumerate}[(i)]
\item  The $n$th power of a star $\mathcal{S}_k^n$: the direct product of $n$ copies of a star $\mathcal{S}_k$.  Note that $\mathcal{S}_k^n$ can be represented as the set of all $n$-tuples $a^n\in \{0,1,\ldots, k\}^n$ where the rank of an element $a^n$ is defined as the number of nonzero coordinates in $a^n$.
\item The linear poset (linear lattice) $\L_n(q)$: the poset of all subspaces of a vector space $GF(q)^n$ ordered by inclusion.
\item The affine poset $ \mathcal{A}_n(q)$: the poset of all affine subspaces of a vector space $GF(q)^n$ ordered by inclusion.  The rank of an element $V$ in both posets is the dimension of $V$.
\item Let $\P$ be a ranked poset and let $l,m$ be integers where $ 0\leq l<m\leq r(\P)$.  Then the poset $\P(l,m):=\P_l\cup\ldots\cup \P_m$  is called a {\em truncated poset} of $\P$. This gives another class of regular, unimodal, level-connected posets obtained from  $\mathcal{S}_k^n, \L_n(q), \mathcal{A}_n(q)$.
\item The product of $n$ chains $C(k_1,\ldots,k_n)$ of sizes $k_1\geq \ldots\geq k_n$ with  $k_2+\ldots+k_n\geq k_1$ is an example of a class of strictly normal unimodal posets, which are not regular (see \cite{Gr}).
\end{enumerate}
For other classes of strictly normal posets see \cite{E3} (Chapter 4.6).

\section{Acknowledgement}
The authors express their gratitude to an unknown referee for showing out an essential error in a previous version of this paper.


\begin{thebibliography} {xxx}

\bibitem{AZ1} R. Ahlswede,  Z. Zhang, An identity in combinatorial extremal theory. {\sl Adv. Math.} {\bf 80} (2) (1990),  137--151.

\bibitem{AZ2} R. Ahlswede,  Z. Zhang, On cloud-antichains and related configurations, {\sl discrete Math.} {\bf 85} (1990), 225-245.

\bibitem{AE} H. Aydinian and P.L. Erd\H{o}s, All maximum size 2--part Sperner systems---in short, {\sl Comb. Prob. Comp.} {\bf 16} (4)     (2007), 553--555.

\bibitem{bak} K.A. Baker, A Generalization of Sperner's Lemma, {\sl J. Comb.Theory} {\bf 6} (1969), 224--225.

\bibitem{Bo} B. Bollob\'as, On generalized graphs,  {\sl Acta Mathematica Academiae Scientiarum Hungaricae} {\bf 16} (1965) (34),  447--452.

\bibitem{E1} K. Engel, Strong properties in partially ordered sets, I. {\sl Discrete Math.} {\bf 47}  (1983), 229--234.

\bibitem{E2} K. Engel,  Strong properties in partially ordered sets, II. {\sl Discrete Math.} {\bf 48}  (1984), 187--196.

\bibitem{E3} K. Engel, {\it Sperner Theory}, Encyclopedia of Mathematics and its Applications, Vol 65, Cambridge Univ. Press, 1997.

\bibitem{erdos} Paul Erd\H os, On a lemma of Littlewood and Offord, {\sl  Bull. of the Amer. Math. Soc.} {\bf 51} (1945), 898--902.

\bibitem{EK3} P.L. Erd\H{o}s and G.O.H. Katona, Convex hulls of more-part Sperner families, {\sl Graphs and Combin.} {\bf 2} (1986), 123--134.

\bibitem{EK2} P.L. Erd\H{o}s and G.O.H. Katona, All maximum 2--part Sperner families, {\sl J. Combin. Theory} (A) {\bf 43} (1986), 58--69.

\bibitem{GH} R.L. Graham and L.H. Harper: Some results on matchings in bipartite graphs, {\sl SIAM J. Appl. Math} {\bf 17} (1969), 1017-1022.

\bibitem{Gr} J.D.  Griggs: Maximum antichains in the product of chains. Order 1 (1984), no. 1, 21-28.

\bibitem{katona} G.O.H. Katona: On a conjecture of Erd\H os and a stronger form  of Sperner's theorem, {\em Studia Sci. Math. Hungar.} {\bf 1} (1966),  59--63.

\bibitem{katona1} G.O.H. Katona, A generalization of some generalizations of
    Sperner's Theorem, {\sl J. Combin. Theory} (B) {\bf 12} (1972), 72--81.

\bibitem{kleitman} D.J. Kleitman, On a lemma of Littlewood and Offord on the distribution of certain sums, {\em Math. Zeit.} {\bf 90} (1965), 251--259.

\bibitem{Kl} D.J. Kleitman,  On an extremal property of antichains in partial orders. The LYM property and some of its implications and  applications,  in {\sl Combinatorics, Part 2: Graph theory; foundations, partitions and combinatorial geometry},   Math. Centre Tracts, No. 56, Math. Centrum,  Amsterdam, (1974), 77--90.

\bibitem{lovasz} L. Lov\'asz, {\sl Combinatorial Problems and Exercises}, 2nd ed., North--Holland, Amsterdam, 1993. Problem 13.21 (b)

\bibitem{Lu} D. Lubell, A short proof of Sperner's lemma, {\sl J. Comb. Theory}
    {\bf 1} (1966) (2), 299.

\bibitem{Me} L.D. Meshalkin, Generalization of Sperner's theorem on the number of subsets of a finite set, {\sl Theory of Probability and its Applications} {\bf 8} (1963) (2), 203--204.

\bibitem{Sa}  A. Sali, A note on convex hulls of more-part Sperner families, {\sl J. Comb. Th.} (A) {\bf 49} (1988), 188--190.

\bibitem{schon} J. Schonheim, A generalization of results of P.  Erd\H{o}s, G. Katona, and D.J. Kleitman concerning Sperner's theorem,  {\sl J. Comb. Theory} (A) {\bf 11} (1971), 111--117.

\bibitem{Sp} E. Sperner, Ein Satz \"uber Untermengen einer endlichen Menge, {\sl Math. Z.} {\bf 27} (1928), 544--548.

\bibitem{Ya} K. Yamamoto, Logarithmic order of free distributive  lattice,  {\sl J. Math. Soc. Japan} {\bf 6} (1954), 343--353.

\end{thebibliography}
\end{document}